\documentclass[11pt]{amsart}
\usepackage{geometry}                
\usepackage{tikz}
\usepackage{amssymb}
\usepackage{amsmath}
\usepackage{enumitem}
\usepackage{hyperref}
\usepackage[pagewise]{lineno}

\newtheorem{Theorem}{Theorem}[section]
\newtheorem{Lemma}[Theorem]{Lemma}
\newtheorem{Proposition}[Theorem]{Proposition}

\newtheorem{Remark}[Theorem]{Remark}

\title[]{Planar closed curves with prescribed curvature}

\author[]{{\normalsize{Paolo Caldiroli\footnote{Email: paolo.caldiroli@unito.it, anna.capietto@unito.it. The first author is member of the Gruppo Nazionale per l'Analisi Matematica, la Probabilit\`a e le loro Applicazioni (GNAMPA) of the Istituto Nazionale di Alta Matematica (INdAM).}, Anna Capietto}}
\vspace{6pt}\\Dipartimento di Matematica, Universit\`a degli Studi di Torino\vspace{3pt}\\via Carlo Alberto 10 -- 10123 Torino, Italy\vspace{3pt}}
\date{}

\begin{document}
\begingroup
\def\uppercasenonmath#1{} 
\let\MakeUppercase\relax 
\maketitle
\endgroup

\begin{abstract}
\noindent
By variational methods, we prove existence of planar closed curves with prescribed curvature, for some classes of curvature functions. The main difficulty is to obtain bounded Palais-Smale sequences. This is achieved by adding a parameter in the problem and using a version of the ``monotonicity trick'' introduced by M. Struwe in \cite{Str88An, Str88AM}.
\smallskip

\noindent
\textit{Keywords:} {Curves in Euclidean space, Prescribed curvature, Variational methods, Lack of compactness.}
\smallskip

\noindent{{{\it 2010 Mathematics Subject Classification:} 53A04, 51M25, 53C42, 47J30}}
\end{abstract}
\section{Introduction}

The prescribed curvature problem for closed curves consists of studying existence (or non existence) of planar closed curves whose curvature at every point agrees with a prescribed value, depending on the point. 

This problem has been investigated in recent years by many authors and with different techniques, see \cite{BetCalGui02, CalCor19, CalGui07, CorMus23, GuiRol10, KirLau10, Mus11, NovVal11}. As a geometrical problem, it constitutes a toy model that could be useful to provide some information for understanding other, more difficult geometrical issues in higher dimension, e.g., the $H$-bubble problem (see the very recent paper \cite{SirWeiZenZhu23} and the references therein). Moreover, regarding applications, it is somehow related to the problem of magnetic geodesics (see, e.g., \cite{Gin04}, and also \cite{CalCor19, CalGui07} for very simple examples). 

Let us state the problem in a more precise form. A planar closed parametric curve is the image $\mathcal{C}$ of a nonconstant, periodic, continuous mapping $u\colon\mathbb{R}\to\mathbb{R}^{2}$. A parametrization $u$ of $\mathcal{C}$ is regular when $u$ is of class $C^{1}$ and $\dot{u}(t)\ne 0$ for every $t\in\mathbb{R}$. If $\mathcal{C}$ admits a regular parametrization $u$ of class $C^{2}$, the curvature of $\mathcal{C}$ is given by
\[
\kappa(t)=\frac{\ddot{u}(t)\cdot i\dot{u}(t)}{|\dot{u}(t)|^{3}}
\]
where $i$ denotes the rotation of $\frac\pi2$. Given a continuous function $K\colon\mathbb{R}^{2}\to\mathbb{R}$, we are looking for planar closed curves $\mathcal{C}$ admitting a regular parametrization $u$ of class $C^{2}$ such that $\kappa(t)=K(u(t))$ for every $t\in\mathbb{R}$. Such curves will be called \emph{$K$-loops}, for short.

When $K$ is a nonzero constant, every circumference with radius $\left|K\right|^{-1}$ is a $K$-loop. However, as soon as $K$ is non constant, existence/non existence, uniqueness/multiplicity of $K$-loops are nontrivial issues and the behaviour of the prescribed curvature function $K$ plays a decisive role. For example, when $K$ is a Lipschitz-continuous positive function depending just on one variable in a strictly monotone way, no $K$-loop exists (see \cite{KirLau10}). On the other hand, for some classes of functions $K$ which converge to a nonzero constant at infinity, existence of $K$-loops is obtained in \cite{Mus11, CorMus23}. Let us also mention the papers \cite{GuiRol10, CalCor19} which deal with the case of symmetric curvature functions, under different assumptions, and with different results and methods (see also  \cite{Mus11}).

In this work we tackle the $K$-loop problem when $K$ is periodic in its two variables. This class of curvature functions is suitable for modeling periodic heterogeneous media. Here is our first main result:

\begin{Theorem}\label{T:main}
Let $K\colon\mathbb{R}^{2}\to\mathbb{R}$ be a continuous function satisfying:
\begin{itemize}
\item[\emph{(K1)}]
there exist $a,b>0$ such that $K(x+a,y)=K(x,y+b)=K(x,y)~\forall(x,y)\in\mathbb{R}^{2}$,
\item[\emph{(K2)}]
$\int_{[0,a]\times[0,b]}K(x,y)\,dx\,dy\ne 0$.
\end{itemize}
Then for almost every $\lambda\in\mathbb{R}\setminus\{0\}$ there exists a $\lambda K$-loop.
\end{Theorem}

The problem of $K$-loops for prescribed periodic curvatures was already considered in \cite{NovVal11}, where for any periodic nonzero function $K$ and for any $\varepsilon>0$ the authors were able to construct a perturbation $K_{\epsilon}$ which is $\varepsilon$-close to $K$ in $L^{1}$, has the same average of $K$ and for which a convex $K_{\varepsilon}$-loop exists. Our result differs substantially, both for the techniques and because we do not modify the curvature function, up to a multiplicative constant in a set of full measure. 

In order to prove Theorem \ref{T:main} we take a variational approach. Let us give more details. As observed in \cite{BetCalGui02}, the $K$-loop problem is equivalent to find 1-periodic, nonconstant solutions $u\colon\mathbb{R}\to\mathbb{R}^{2}$ of 
\begin{equation}\label{eqK}
\ddot{u}=L(u)K(u)i\dot{u}
\end{equation}
where 
\[
L(u):=\left(\int_{0}^{1}|\dot{u}|^{2}\,dt\right)^{\frac12}.
\]
The search for nonconstant 1-periodic solutions of \eqref{eqK} is a variational problem. Indeed, let us consider the Sobolev space 
\begin{equation}\label{eq:H1}
H^{1}:=H^{1}(\mathbb{R}/\mathbb{Z},\mathbb{R}^{2})
\end{equation}
endowed with the Hilbertian norm
\begin{equation}\label{eq:norm}
\|u\|^{2}=\int_{0}^{1}|\dot{u}|^{2}\,dt+\left|\int_{0}^{1}u\,dt\right|^{2}.
\end{equation}
Let $Q\colon\mathbb{R}^{2}\to\mathbb{R}^{2}$ be the vector field defined by
\begin{equation}
\label{Q}
Q(x,y)=\frac12\left(\int_{0}^{x}K(s,y)\,ds\,, \int_{0}^{y}K(x,s)\,ds\right)\quad\forall (x,y)\in\mathbb{R}^{2}
\end{equation}
and let 
\begin{equation}\label{eqG}
E(u):=L(u)+G(u)\quad\text{where}\quad G(u):=\int_{0}^{1}Q(u)\cdot i\dot{u}\,dt~\forall u\in H^{1}\,.
\end{equation}
The energy functional $E$ turns out to be continuous in $H^{1}$ and of class $C^{1}$ in $H^{1}\setminus\mathbb{R}^{2}$ (here we identify $\mathbb{R}^{2}$ with the space of constant functions), and  
$u$ is a nonconstant (classical) 1-periodic solution of \eqref{eqK} if and only if $u\in H^{1}\setminus\mathbb{R}^{2}$ is a critical point of $E$ (see Proposition \ref{P:regularity}). 
The assumptions on $K$ guarantee that  $E$ has a mountain-pass geometry. In particular, the value
\[
c:=\inf_{\gamma\in\Gamma}\max_{s\in[0,1]}E(\gamma(s))\quad\text{where}\quad\Gamma:=\{\gamma\in C([0,1],H^{1})\colon\gamma(0)=0\,,~E(\gamma(1))<0\}
\]
is positive and is an asymptotic critical value. This means that there exists a Palais-Smale sequence (PS sequence, for short) for $E$ at level $c$, that is, a sequence $(u_{n})\subset H^{1}$ such that $E(u_{n})\to c$ and $E'(u_{n})\to 0$ in the dual of $H^{1}$. In fact, the main difficulty is to show that $E$ admits \emph{at least one bounded} PS sequence at a nonzero level. In this regard, we point out that the norm \eqref{eq:norm} is made up of two terms which essentially measure the length and the barycenter of loops. The periodicity of $K$ and a simple translation argument allows us to obtain plainly PS sequences with bounded barycenter. As far as concerns the $L^{2}$-norm of the derivative, we are not able to obtain a bound for a PS sequence of $E$. In order to overcome this difficulty we use a version of the ``monotonicity trick'' introduced by M. Struwe in \cite{Str88An, Str88AM}. More precisely, we consider a family of problems 
\begin{equation}
\label{lambda-eq}
\ddot{u}=\lambda L(u)K(u)i\dot{u}\,,
\end{equation}
where $\lambda$ is a real parameter. For every $\lambda\in\mathbb{R}$ we introduce the energy functional corresponding to \eqref{lambda-eq}, defined by
\begin{equation}\label{energy}
E_{\lambda}(u):=L(u)+\lambda G(u)\quad\forall u\in H^{1}\,,
\end{equation}
with $G$ as in \eqref{eqG}. For every $\lambda\ne 0$ the functional $E_\lambda$ admits a mountain-pass geometry, with mountain-pass level $c_\lambda$. It turns out that whenever the upper left Dini derivative of the mapping $\lambda\mapsto c_{\lambda}$ is finite or $+\infty$ at $\lambda_0$ (in this case, we say that the Denjoy property for $c_\lambda$ holds true at $\lambda_0$), then it is possible to construct a PS sequence $(u_{n})\subset H^{1}$ for $E_{\lambda_0}$ at level $c_{\lambda_0}$ with $\sup_{n}L(u_{n})<\infty$. With this additional information and in view of the periodicity of $K$, after some standard work, one can plainly conclude that $c_{\lambda_0}$ is a critical value and then that \eqref{lambda-eq} with $\lambda=\lambda_0$ admits a nonconstant periodic solution, that is, a $\lambda_0 K$-loop exists. Then, the Denjoy-Young-Saks Theorem \cite{Sak64} ensures that the Denjoy property holds true for almost every $\lambda\in\mathbb{R}\setminus\{0\}$. In this way we prove Theorem \ref{T:main}.

The same strategy can be used to consider also prescribed curvature problems for closed curves for a class of curvature functions which converge to a nonzero constant at infinity. More precisely, we can show:

\begin{Theorem}\label{T:2}
Let $K\colon\mathbb{R}^{2}\to\mathbb{R}$ be a continuous function satisfying:
\begin{itemize}
\item[$(K4)$]
$K(x,y)=K_0+K_1(x,y)$ with $K_0\in\mathbb{R}\setminus\{0\}$ and $K_1(x,y)\to 0$ as $|(x,y)|\to\infty$,
\item[$(K5)$]
there exist $r>0$ and a nonempty open interval $I\subset\mathbb{S}^1$ such that $K_0K_1(z)>0$ for $|z|>r$ and $\mathrm{arg}(z)\in I$.
\end{itemize}
Then for almost every $\lambda\in\mathbb{R}\setminus\{0\}$ there exists a $\lambda K$-loop.
\end{Theorem}

In fact, a more general version of Theorem \ref{T:2} holds true, see Theorem \ref{T3}, but under less explicit assumptions on $K$. Theorems \ref{T:2} and \ref{T3} can be compared to some existence results obtained in \cite[Theorem 2.5]{Mus11} and \cite{CorMus23}. The main differences are that we ask no condition on the rate decay at infinity of the ``nonconstant'' part $K_{1}$, neither any other global condition. On the other hand, we cannot handle the exact case when $\lambda=1$ and, taking a sequence of $\lambda_{n}K$-loops $u_{n}$ with $\lambda_{n}\to 1$ and bounded energies, we are not able to control in a uniform way the winding number of $u_{n}$, i.e., the degree of the mappings $\frac{\dot{u}_{n}}{|\dot{u}_{n}|}$. This additional information would be useful to pass to the limit in order to obtain existence of a $K$-loop, but we suspect that further assumptions on $K$ are necessary to this aim.

\section{The mountain pass geometry}
As already observed in the Introduction, for every $\lambda\in\mathbb{R}$ the functional $E_{\lambda}$ defined in \eqref{energy} is well defined and regular enough in the Hilbert space $H^{1}$ introduced in \eqref{eq:H1}. More precisely:

\begin{Proposition}\label{P:regularity}
Assume that $K\colon\mathbb{R}^{2}\to\mathbb{R}$ is a continuous function and let $E_\lambda\colon H^1\to\mathbb{R}$ be defined as in \eqref{energy}. Then $E_\lambda\in C(H^{1})\cap C^1(H^{1}\setminus\mathbb{R}^{2})$ and  
\begin{equation}\label{E'}
E'_\lambda(u)[h]=L(u)^{-1}\int_{0}^{1}\dot u\cdot\dot h\,dt+\lambda\int_{0}^{1}K(u)h\cdot i\dot u\,dt\quad\forall u\in H^{1}\setminus\mathbb{R}^{2},~\forall h\in H^{1}.
\end{equation}
In particular $u$ is a nonconstant (classical) 1-periodic solution of \eqref{lambda-eq} if and only if $u\in H^{1}\setminus\mathbb{R}^{2}$ is a critical point of $E_\lambda$, i.e. $E'_\lambda(u)=0$.
\end{Proposition}

\begin{proof}
When $K$ is of class $C^{1}$, also the vector field $Q$ defined in \eqref{Q} is so and satisfies $\mathrm{div}(Q)=K$ on $\mathbb{R}^{2}$. Then, one can argue as in \cite{CalGui07}. When $K$ is only continuous, one follows an approximation argument as in \cite[Lemma 1.3.3]{CorMus23} and concludes. 
\end{proof}

Now, we show that $E_{\lambda}$ admits a mountain pass geometry. More precisely, we prove:

\begin{Proposition}\label{P:mp}
Assume that $K\colon\mathbb{R}^{2}\to\mathbb{R}$ is a continuous function satisfying $(K1)$-$(K2)$. Then, for every $[\alpha,\beta]\subset\mathbb{R}\setminus\{0\}$ there exists $\overline{u}\in H^{1}$ such that 
\begin{equation}\label{baru}
E_{\lambda}(\overline{u})\le 0\quad\text{and}\quad L(\overline{u})>\frac{2\pi}{\left|\lambda\right|\left\|K\right\|_{\infty}}~\forall\lambda\in[\alpha,\beta]
\end{equation}
Moreover, setting $\Gamma:=\{\gamma\in C([0,1],H^{1})\colon\gamma(0)=0\,,~\gamma(1)=\overline{u}\}$ and 
\[
c_{\lambda}:=\inf_{\gamma\in\Gamma}\max_{s\in[0,1]}E_{\lambda}(\gamma(s))\,,
\]
one has that 
\[
c_{\lambda}\ge\frac{\pi}{\left|\lambda\right|\left\|K\right\|_{\infty}}~\forall\lambda\in[\alpha,\beta]\,.
\]
\end{Proposition}

Let us start by stating a version of isoperimetric inequality.
\begin{Theorem}[weighted isoperimetric inequality]\label{ii} 
Assume that $K\colon\mathbb{R}^{2}\to\mathbb{R}$ is a continuous, bounded function and let $G\colon H^{1}\to\mathbb{R}$ be the functional defined in \eqref{energy}. Then
\begin{equation}\label{eq:ii}
|G(u)|\le\frac{\|K\|_{\infty}}{4\pi}\left(\int_{0}^{1}|\dot{u}|\,dt\right)^{2}\le \frac{\|K\|_{\infty}}{4\pi}\,L(u)^{2}\quad\forall u\in H^{1}\,.
\end{equation}
\end{Theorem}

\begin{proof}
Let $u\colon\mathbb{R}/\mathbb{Z}\to\mathbb{R}^{2}$ be a continuous, 1-periodic function, with piecewise continuous derivative. Let $\mathcal{C}_{u}:=\mathrm{range}(u)$. The order of a point $(x,y):=z\in\mathcal{C}_{u}$ is defined by $\#\{t\in[0,1)\colon u(t)=z\}$. A point $z\in\mathcal{C}_{u}$ is a self-intersection point when its order is larger than 1. Assume that $u$ has isolated self-intersections of finite order. 
Let us denote $\mathrm{Ind}_{u}(x,y)$ the index (or winding number) of the curve $\mathcal{C}_{u}:=\mathrm{range}(u)$ with respect to a point $(x,y)\colon\mathbb{R}^{2}\setminus\mathcal{C}_{u}$. In complex notation 
\[
\mathrm{Ind}_{u}(z)=\frac{1}{2\pi i}\int_{\mathcal{C}_{u}}\frac{d\zeta}{\zeta-z}=-\frac{1}{2\pi}\int_{0}^{1}\frac{(u-z)\cdot i\dot{u}}{|u-z|^{2}}\,dt\,.
\]
It is well known that the map $z\mapsto\mathrm{Ind}_{u}(z)$ takes integer values, is constant on each connected component of $\mathbb{R}^{2}\setminus\mathcal{C}_{u}$, and vanishes on the unbounded component of $\mathbb{R}^{2}\setminus\mathcal{C}_{u}$. Moreover, as proved in \cite{KirLau10},
\begin{equation}\label{KirschLaurain}
G(u)=-\int_{\mathbb{R}^{2}}\mathrm{Ind}_{u}(x,y)K(x,y)\,dx\,dy\,.
\end{equation}
In addition, 
\begin{equation}\label{Rado}
\int_{\mathbb{R}^{2}}|\mathrm{Ind}_{u}(x,y)|\,dx\,dy\le\frac1{4\pi}\left(\int_{0}^{1}|\dot u|\,dt\right)^{2}
\end{equation}
(see \cite{Rad47}). Then \eqref{KirschLaurain}--\eqref{Rado} plainly imply \eqref{eq:ii}.  If $u\in H^{1}$, for every $n\in\mathbb{N}$, one can easily construct a sequence of piecewise linear mappings $u_{n}\colon\mathbb{R}/\mathbb{Z}\to\mathbb{R}^{2}$ which parameterize polygonals with isolated self-intersections of finite order and such that $u_{n}\to u$ in $H^{1}$. Since \eqref{eq:ii} holds true for each $u_{n}$, passing to the limit, one obtains \eqref{eq:ii} for $u$, too. 
\end{proof}

\begin{Lemma}\label{L:min}
If $K\colon\mathbb{R}^2\to\mathbb{R}$ is a continuous, bounded function, then for every $\lambda\in\mathbb{R}$ and for every $u\in H^{1}$ one has that
\[
L(u)\le \frac{2\pi}{|\lambda|\|K\|_{\infty}}\Rightarrow E_{\lambda}(u)\ge\frac{L(u)}{2}\,.
\]
\end{Lemma}

\begin{proof}
Fix $\lambda\in\mathbb{R}\setminus\{0\}$. In view of the isoperimetric inequality \eqref{eq:ii}, for every $u\in H^{1}$ one has that
\[
E_{\lambda}(u)\ge L(u)-\left|\lambda\right|\left|G(u)\right|\ge L(u)\left(1-\frac{\left|\lambda\right|\|K\|_{\infty}}{4\pi}\,L(u)\right)
\]
and the conclusion follows. 
\end{proof}

\begin{Lemma}\label{L:rect}
If $K\colon\mathbb{R}^2\to\mathbb{R}$ is a continuous function satisfying $(K1)$, then for every $n\in\mathbb{N}$ there exist piecewise linear parametrizations $u_{n}^{+},u_{n}^{-}\in H^{1}$ of the boundary of the rectangle $[0,na]\times[0,nb]$, with counter-clockwise, or clockwise orientation, respectively, such that
\begin{equation}\label{eq:rect}
L(u_{n}^{+})=2n(a+b)\quad\text{and}\quad E_{\lambda}(u^{\pm}_{n})=2n(a+b)\mp n^{2}\lambda\int_{[0,a]\times[0,b]}K(x,y)\,dx\,dy\,.
\end{equation}
\end{Lemma}

\begin{proof} Let us denote $\mathbf{e}_{i}$ the $i$-th vector of the canonical basis of $\mathbb{R}^{2}$, $i=1,2$, let
\[
u_{n}^{+}(t):=\begin{cases}2(a+b)nt\mathbf{e}_{1}&\text{as $0\le t\le\frac a{2(a+b)}$}\\na(\mathbf{e}_{1}-\mathbf{e}_{2})+2(a+b)nt\mathbf{e}_{2}&\text{as $\frac a{2(a+b)}<t\le\frac12$}\\(2a+b)n\mathbf{e}_{1}+bn\mathbf{e}_{2}-2(a+b)nt\mathbf{e}_{1}&\text{as $\frac12<t\le\frac12+\frac{a}{2(a+b)}$}\\2(a+b)(1-t)n\textbf{e}_{2}&\text{as $\frac12+\frac{a}{2(a+b)}<t\le 1$}
\end{cases}
\]
and let us extend $u_{n}^{+}$ on $\mathbb{R}$ in a periodic way, with period 1. One easily checks that $u_{n}^{+}\in H^{1}$, $u_{n}^{+}$ is a parametrization of the boundary of the rectangle $[0,na]\times[0,nb]$, with counter-clockwise orientation, 
\[
L(u_{n}^{+})=2n(a+b)\quad\text{and}\quad G(u_{n}^{+})=-n^{2}\int_{[0,a]\times[0,b]}K(z)\,dz\,.
\] 
Hence \eqref{eq:rect} is proved for $u_n^+$. Finally, the mapping $u_{n}^{-}(t):=u_{n}^{+}(-t)$ is in $H^{1}$, is a parametrization of the boundary of the same rectangle, with clockwise orientation, $L(u_{n}^{-})=L(u_{n}^{+})$, and $G(u_{n}^{-})=-G(u_{n}^{+})$. Hence  \eqref{eq:rect} holds also for $u_n^+$. 
\end{proof}

\noindent
\emph{Proof of Proposition \ref{P:mp}.} Fix $\alpha,\beta>0$ with $\alpha<\beta$. We use the assumption (K2). Assume, in particular, that $\int_{[0,a]\times[0,b]}K(x,y)\,dx\,dy>0$. By Lemma \ref{L:rect}, there exists $\overline{n}\in\mathbb{N}$ such that $L(u_{\overline{n}}^{+})>\frac{2\pi}{\left|\lambda\right|\|K\|_{\infty}}$ and $E_{\lambda}(u_{\overline{n}}^{+})\le0$ for every $\lambda\in[\alpha,\beta]$. Set $\overline{u}:=u_{\overline{n}}^{+}$. Then for every $\gamma\in C([0,1],H^{1})$ such that $\gamma(0)=0$ and $\gamma(1)=\overline{u}$, by continuity, there exists $\overline{s}\in(0,1)$ such that $L(\gamma(\overline{s}))=\frac{2\pi}{|\lambda|\|K\|_{\infty}}$. Thence, by Lemma \ref{L:min}
\[
\max_{s\in[0,1]}E_{\lambda}(\gamma(s))\ge E_{\lambda}(\gamma(\overline{s}))\ge \frac{\pi}{|\lambda|\|K\|_{\infty}}
\]
and the conclusion plainly follows. In case $\int_{[0,a]\times[0,b]}K(x,y)\,dx\,dy<0$ one takes $\overline{u}:=u_{\overline{n}}^{-}$. If $\alpha<\beta<0$, one argues in a similar way.\hfill$\square$

\begin{Remark}
The assumption $(K2)$ is used just to guarantee the existence of some $\overline{u}\in H^1$ satisfying \eqref{baru}. Actually \eqref{baru} might hold true also in other cases, even if the average of $K$ on the periodicity rectangle is null. Suppose for instance that 
\begin{itemize}
\item[$(K3)$] $K(x_0,y_0)\ne 0$ at some $(x_0,y_0)\in\mathbb{R}^2$.
\end{itemize}
In particular, if  $K(x_0,y_0)> 0$, then for $r>0$ small enough, $K(x,y)\ge\frac{K(x_0,y_0)}{2}>0$ in a disc $D_r(x_0,y_0)$. Let $\overline{u}(t)=u_0+re^{2\pi it}$. Then $L(\overline{u})=2\pi r$ and $E_\lambda(\overline{u})=2\pi r-\lambda \pi r^2 \frac{K(x_0,y_0)}{2}$. Therefore \eqref{baru} holds true for $\lambda>0$ sufficiently large. Many variations of this argument can be developed. 
\end{Remark}

\section{Palais-Smale sequences with bounded length}
Let us fix an interval $[\alpha,\beta]\subset\mathbb{R}\setminus\{0\}$ and for every $\lambda\in[\alpha,\beta]$ let $c_{\lambda}$ be the mountain pass level of $E_{\lambda}$ (see Proposition \ref{P:mp}). In this section we assume that the mapping $\lambda\mapsto c_{\lambda}$ satisfies the Denjoy property at some $\lambda_{0}\in(\alpha,\beta)$, that is: 
\begin{itemize}
\item[$(D)$] there exists a sequence $(\lambda_{n})_{n\in\mathbb{N}}\subset(\alpha,\beta)$ such that
\[
\lambda_{n}<\lambda_{n+1}~~\forall n\in\mathbb{N}\,,\quad
\lambda_{n}\to\lambda_{0}\text{ as }n\to\infty\,,\quad
\sup_{n\in\mathbb{N}}\frac{c_{\lambda_{n}}-c_{\lambda_{0}}}{\lambda_{0}-\lambda_{n}}<\infty\,.
\]
\end{itemize}
Our goal is to show:
\begin{Proposition}\label{P:PS-bound}
Let $\lambda_{0}\in(\alpha,\beta)$ be such that $(D)$ holds. Then the functional $E_{\lambda_{0}}$ admits a PS sequence $(u_{n})_{n}\subset H^{1}$ at level $c_{\lambda_{0}}$ such that $\sup_{n}L(u_{n})<\infty$. 
\end{Proposition}

The proof of Proposition \ref{P:PS-bound} uses the monotonicity trick conceived by M. Struwe in \cite{Str88An, Str88AM}. This method, in its original form, exploits the almost everywhere differentiability of monotone functions (hence the name) and allows us to obtain some bound on the PS sequences. It has found several applications for different variational problems (see \cite{Str22}) and many versions and refinements have been developed. Here we follow essentially a version due to Jeanjean and Toland \cite{JeaTol98}, which works when one has no information on monotonicity or almost everywhere differentiability of the function $\lambda\mapsto c_{\lambda}$. Actually, as pointed out in \cite{JeaTol98}, the Denjoy property (D) is enough. 

\begin{Lemma}\label{L:bound}
Let $\lambda_{0}\in(\alpha,\beta)$ and let $(\lambda_{n})_{n\in\mathbb{N}}$ be a sequence satisfying $(D)$. Then there exists $(\gamma_{n})_{n\in\mathbb{N}}\subset\Gamma$ and a constant $C_{0}>0$ (depending on $\lambda_{0}$) such that:
\begin{itemize}
\item[\emph{(i)}] if $E_{\lambda_{0}}(\gamma_{n}(s))\ge c_{\lambda_{0}}-(\lambda_{0}-\lambda_{n})$ then $L(\gamma_{n}(s))\le C_{0}$,
\item[\emph{(ii)}] $\max_{s\in[0,1]}E_{\lambda_{0}}(\gamma_{n}(s))\to c_{\lambda_{0}}$ as $n\to\infty$.
\end{itemize}
\end{Lemma}
\begin{proof}
For every $n\in\mathbb{N}$ there exists $\gamma_{n}\in\Gamma$ such that 
\begin{equation}\label{gamman}
\max_{s\in[0,1]}E_{\lambda_{n}}(\gamma_{n}(s))\le c_{\lambda_{n}}+\lambda_{0}-\lambda_{n}\,.
\end{equation}
If $E_{\lambda_{0}}(\gamma_{n}(s))\ge c_{\lambda_{0}}-(\lambda_{0}-\lambda_{n})$, then
\[\begin{split}
-G(\gamma_{n}(s))&=\frac{E_{\lambda_{n}}(\gamma_{n}(s))-E_{\lambda_{0}}(\gamma_{n}(s))}{\lambda_{0}-\lambda_{n}}\\
&\le\frac{c_{\lambda_{n}}+(\lambda_{0}-\lambda_{n})-c_{\lambda_{0}}+(\lambda_{0}-\lambda_{n})}{\lambda_{0}-\lambda_{n}}\le\sup_{n\in\mathbb{N}}\frac{c_{\lambda_{n}}-c_{\lambda_{0}}}{\lambda_{0}-\lambda_{n}}+2:=C_{1}
\end{split}\]
thence
\[
L(\gamma_{n}(s))=E_{\lambda_{n}}(\gamma_{n}(s))-\lambda_{n} G(\gamma_{n}(s))\le c_{\lambda_{n}}+\lambda_{0}-\lambda_{n}+\lambda_{n} C_{1}\le C_{0}
\]
Thus (i) is proved. Let us discuss (ii). For every $n\in\mathbb{N}$ there exists $s_{n}\in[0,1]$ such that
\[
E_{\lambda_{0}}(\gamma_{n}(s_{n}))=\max_{s\in[0,1]}E_{\lambda_{0}}(\gamma_{n}(s))\,.
\]
Setting $u_{n}=\gamma_{n}(s_{n})$, by \eqref{eq:ii} and \eqref{gamman} we estimate
\begin{align}
\nonumber
E_{\lambda_{0}}(u_{n})&=E_{\lambda_{n}}(u_{n})+(\lambda_{0}-\lambda_{n})G(u_{n})\le\max_{s\in[0,1]}E_{\lambda_{n}}(\gamma_{n}(s))+(\lambda_{0}-\lambda_{n})\frac{\left\|K\right\|_{\infty}}{4\pi}L(u_{n})^{2}\\
\nonumber
&\le c_{\lambda_{n}}+(\lambda_{0}-\lambda_{n})\left[1+\frac{\left\|K\right\|_{\infty}}{4\pi}L(u_{n})^{2}\right]\\ 
\label{leq}
&=c_{\lambda_{0}}+(\lambda_{0}-\lambda_{n})\left[1+\frac{c_{\lambda_{n}}-c_{\lambda_{0}}}{\lambda_{0}-\lambda_{n}}+\frac{\left\|K\right\|_{\infty}}{4\pi}L(u_{n})^{2}\right]\,.
\end{align}
Since 
\begin{equation}\label{geq}
E_{\lambda_{0}}(u_{n})=\max_{s\in[0,1]}E_{\lambda_{0}}(\gamma_{n}(s))\ge c_{\lambda_{0}}>c_{\lambda_{0}}-(\lambda_{0}-\lambda_{n})\,,
\end{equation}
by (i), $L(u_{n})\le C_{0}$. Therefore, by $(D)$, from \eqref{leq}--\eqref{geq} it follows that
\[
c_{\lambda_{0}}\le\max_{s\in[0,1]}E_{\lambda_{0}}(\gamma_{n}(s))\le c_{\lambda_{0}}+o(1)\,.
\]
Hence, also (ii) is proved.
\end{proof}

\begin{Lemma}\label{L:Meps}
If $\lambda_{0}\in(\alpha,\beta)$ is such that $(D)$ holds, and $C_{0}$ is the constant (depending on $\lambda_{0}$) given by Lemma \ref{L:bound}, then for every $\varepsilon>0$ the set $M_{\varepsilon}:=\{u\in H^{1}\colon L(u)\le C_{0}+1\,,~|E_{\lambda_{0}}(u)-c_{\lambda_{0}}|\le\varepsilon\}$ is nonempty and $\inf\{\|E'_{\lambda_{0}}(u)\|\colon u\in M_{\varepsilon}\}=0$.
\end{Lemma}
\begin{proof} Fixing $\varepsilon>0$, by Lemma \ref{L:bound}, there exists $\gamma\in\Gamma$ and $s\in[0,1]$ such that $c_{\lambda_{0}}-\varepsilon\le E_{\lambda_{0}}(\gamma(s))\le c_{\lambda_{0}}+\varepsilon$ and then $L(\gamma(s))\le C_{0}$. Hence $M_{\varepsilon}\ne\varnothing$. Let us prove the second part of the Lemma. Arguing by contradiction, assume that there exists $\varepsilon_{0}\in\left(0,\frac{c_{\lambda_{0}}}{2}\right)$ such that
\[
\|E'_{\lambda_{0}}(u)\|\ge\varepsilon_{0}\quad\forall u\in M_{\varepsilon_{0}}\,.
\]
Then, by a standard deformation argument, there exist $\varepsilon_{1}\in(0,\varepsilon_{0})$ and a homeomorphism $\eta\colon H^{1}\to H^{1}$ such that
\begin{gather}
\label{eta1}
E_{\lambda_{0}}(\eta(u))\le E_{\lambda_{0}}(u)\quad\forall u\in H^{1}\\
\label{eta2}
\eta(u)=u\quad\text{if }|E_{\lambda_{0}}(u)-c_{\lambda_{0}}|\ge\varepsilon_{0}\\
\label{eta3}
E_{\lambda_{0}}(\eta(u))\le c_{\lambda_{0}}-\varepsilon_{1}\quad\text{if }E_{\lambda_{0}}(u)<c_{\lambda_{0}}+\varepsilon_{1}~~\text{and}~~L(u)\le C_{0}\,.
\end{gather}
Let $(\lambda_{n})_{n\in\mathbb{N}}\subset(\alpha,\beta)$ be a sequence satisfying $(D)$ and let $(\gamma_{n})_{n\in\mathbb{N}}\subset\Gamma$ be given by Lemma \ref{L:bound}.  Let $\widetilde{\gamma}_{n}=\eta\circ\gamma_{n}\in C([0,1],H^{1})$. When $E_{\lambda_{0}}(\gamma_{n}(s))\le c_{\lambda_{0}}-\varepsilon_{0}$ then, by \eqref{eta2}, $\widetilde\gamma_{n}(s)=\gamma_{n}(s)$ and, consequently, also $E_{\lambda_{0}}(\widetilde\gamma_{n}(s))\le c_{\lambda_{0}}-\varepsilon_{1}$. In particular this holds true for $s=0$ since $E_{\lambda_{0}}(\gamma_{n}(0))=E_{\lambda_{0}}(0)=0$ and $c_{\lambda_{0}}-\varepsilon_{0}>\frac{c_{\lambda_{0}}}{2}>0$. The same occurs at $s=1$, since $E_{\lambda_{0}}(\gamma_{n}(1))=E_{\lambda_{0}}(\overline{u})<0$ (Proposition \ref{P:mp}). Thus $\widetilde\gamma_{n}\in \Gamma$. For $n$ large enough $\max_{s\in[0,1]}E_{\lambda_{0}}(\gamma_{n}(s))<c_{\lambda_{0}}+\varepsilon_{1}$. In particular, when $E_{\lambda_{0}}(\gamma_{n}(s))\ge c_{\lambda_{0}}-(\lambda_{0}-\lambda_{n})$, then $L(\gamma_{n}(s)\le C_{0}$ and, by \ref{eta3}, $E_{\lambda_{0}}(\widetilde\gamma_{n}(s))\le c_{\lambda_{0}}-\varepsilon_{1}$. Otherwise, by \eqref{eta1},  $E_{\lambda_{0}}(\widetilde\gamma_{n}(s))\le E_{\lambda_{0}}(\gamma_{n}(s))< c_{\lambda_{0}}-(\lambda_{0}-\lambda_{n})<c_{\lambda_{0}}$. Thence $\max_{s\in[0,1]}E_{\lambda_{0}}(\widetilde\gamma_{n}(s))<c_{\lambda_{0}}$, contrary to the definition of $c_{\lambda_{0}}$. 
\end{proof}

\noindent
\emph{Proof of Proposition \ref{P:PS-bound}.} It plainly follows from Lemma \ref{L:Meps}, taking a sequence $\varepsilon_{n}\to 0^{+}$.

\section{Proof of Theorem \ref{T:main}}

Fix an interval $[\alpha,\beta]\subset\mathbb{R}\setminus\{0\}$ and set 
\begin{equation}
\label{LambdaD}
\Lambda_{\alpha,\beta}:=\{\lambda_{0}\in(\alpha,\beta)\colon \lambda_{0}\text{ satisfies (D)}\}
\end{equation} 
For every $\lambda_{0}\in \Lambda_{\alpha,\beta}$, by Proposition \ref{P:PS-bound}, there exists a PS sequence $(u_{n})_{n}\subset H^1\setminus\mathbb{R}^2$ for $E_{\lambda_{0}}$ at level $c_{\lambda_{0}}$, with $\sup_{n}L(u_{n})<\infty$. For every $n\in\mathbb{N}$ there exists $m_{n}\in\mathbb{Z}^{2}$ such that
\[
\int_{0}^{1}(u_{n}+m_{n,1}a\mathbf{e}_{1}+m_{n,2}b\mathbf{e}_{2})\,dt\in[0,a]\times[0,b]\quad\forall n\in\mathbb{N}\,.
\]
Since, in general, for every $u\in  H^1\setminus\mathbb{R}^2$ and for every $m_1,m_2\in\mathbb{Z}$,
\begin{gather*}
L(u+m_{1}a\mathbf{e}_{1}+m_{2}b\mathbf{e}_{2})=L(u)\,,\\G(u+m_{1}a\mathbf{e}_{1}+m_{2}b\mathbf{e}_{2})=G(u)\,,\\\|E'_{\lambda}(u+m_{1}a\mathbf{e}_{1}+m_{2}b\mathbf{e}_{2})\|=\|E'_{\lambda}(u)\|\,,
\end{gather*}
setting
\[
v_n:=u_{n}+m_{n,1}a\mathbb{e}_{1}+m_{n,2}a\mathbb{e}_{2}
\]
the sequence $(v_n)_{n}$ is a PS sequence for $E_{\lambda_{0}}$ at level $c_{\lambda_{0}}$, and it is bounded in $H^1$. Passing to a subsequence, if necessary, we can assume that $v_{n}\rightharpoonup u$ weakly in $H^{1}$. Then we apply:
\begin{Lemma}\label{L:compactness}
If $(u_{n})\subset H^{1}\setminus\mathbb{R}^{2}$ satisfies $\|E'_{\lambda}(u_{n})\|\to 0$ and $u_{n}\rightharpoonup u$ weakly in $H^{1}$, then $u_{n}\to u$ strongly in $H^{1}$.
\end{Lemma}
Let us complete the proof of Theorem \ref{T:main}. By Lemma \ref{L:compactness} $v_n\to u$ strongly in $H^{1}$. Then $E_{\lambda_0}=c_{\lambda_0}$ and $E'_{\lambda_0}(u)=0$. By Proposition \ref{P:regularity}, $u$ is a classical, 1-periodic solution of  \eqref{lambda-eq} with $\lambda=\lambda_0$ and is nonconstant because $c_{\lambda_0}\ne 0$. Thus $u$ is a parametrization of a $\lambda_0 K$-loop. . By the Denjoy-Young-Saks Theorem \cite{Sak64},  for every interval $[\alpha,\beta]\subset\mathbb{R}\setminus\{0\}$ the set $\Lambda_{\alpha,\beta}$ has full measure in $[\alpha,\beta]$. Therefore, we conclude that for almost every $\lambda\in\mathbb{R}$ there exists a $\lambda K$-loop at energy level $c_\lambda$. 
\medskip

\noindent
\emph{Proof of Lemma \ref{L:compactness}.} The proof is standard (see, e.g. \cite{Mus11,CorMus23}). We sketch it for the reader's convenience. Set
\[
\int_{0}^{1}u\,dt:=\overline{u}\quad\text{and}\quad \int_{0}^{1}u_{n}\,dt:=\overline{u}_{n}~\forall n\in\mathbb{N}\,.
\]
By the compact embedding of $H^{1}$ into $L^{1}$, $\overline{u}_{n}\to\overline{u}$ in $\mathbb{R}^{2}$. If $L(u_{n})\to 0$ then 
\[
\|u_{n}-\overline{u}\|^{2}=L(u_{n})^{2}+|\overline{u}_{n}-\overline{u}|^{2}\to 0
\]
and hence $u=\overline{u}$ and the lemma is proved. If $L(u_{n})\not\to 0$, it is not restrictive to assume that $L(u_{n})\to\ell\in(0,\infty)$. Then, by \eqref{E'} and since $u_{n}\rightharpoonup u$ weakly in $H^{1}$
\[\begin{split}
\int_{0}^{1}|\dot{u}_{n}-\dot{u}|^{2}\,dt&=\int_{0}^{1}\dot{u}_{n}\cdot(\dot{u}_{n}-\dot{u})\,dt+o(1)\\
&=\ell\left(E'_{\lambda}(u_{n})[u_{n}-u]-\lambda\int_{0}^{1}K(u_{n})(u_{n}-u)\cdot i\dot{u}_{n}\,dt\right)+o(1)\,.
\end{split}\]
Since $\|E'_{\lambda}(u_{n})\|\to 0$ and $(u_{n})$ is bounded in $H^{1}$, $E'_{\lambda}(u_{n})[u_{n}-u]\to 0$. Moreover, using the compact embedding of $H^{1}$ into $L^{2}$ and the fact that $K$ is bounded, $K(u_{n})(u_{n}-u)\to 0$ in $L^{2}$ and then, since $(\dot{u}_{n})$ is bounded in $L^{2}$,
\[
\int_{0}^{1}K(u_{n})(u_{n}-u)\cdot i\dot{u}_{n}\,dt\to 0\,.
\]
Thence $\dot{u}_{n}\to\dot{u}$ in $L^{2}$ and the conclusion follows also in this case.\hfill$\square$

\section{The case of prescribed curvature constant at infinity}
In this Section we consider the case in which $K\colon\mathbb{R}^2\to\mathbb{R}$ is a continuous function satisfying (K4).
We set
\[
G_0(u)=\frac12\int_0^1u\cdot i\dot u\,dt\quad\text{and}\quad G_1(u)=\frac12\int_0^1Q_1(u)\cdot i\dot u\,dt
\]
with $Q_1\in C^1(\mathbb{R}^2,\mathbb{R}^2)$ such that $\mathrm{div}(Q_1)=K_1$. Then we define
\begin{equation}\label{E0l}
E_{0,\lambda}(u):= L(u)+\lambda K_0 G_0(u)\,,\quad E_\lambda(u):=E_{0,\lambda}(u)+\lambda G_1(u)\quad\forall u\in H^1\,.
\end{equation}
Also in this case, our first goal is to show that $E_\lambda$ has a mountain pass geometry as soon as $\lambda\ne 0$. It is not restrictive to assume $\lambda>0$, otherwise we change $K$ with $-K$. We can always use Lemma \ref{L:min} but not Lemma \ref{L:rect}. In its place, we have:

\begin{Lemma}\label{L:negative}
If $K\colon\mathbb{R}^2\to\mathbb{R}$ is a continuous function satisfying $(K4)$ with $K_0>0$, then for every $\alpha,\beta>0$ with $\alpha<\beta$ there exist $r_0,R>0$ such that, setting $u_0(t)=r_0e^{2\pi i t}$, one has
\begin{equation}\label{eq:negative}
L(u_0+z)>\frac{2\pi}{\lambda\left\|K\right\|_\infty}~~\text{and}~~ E_{\lambda}(u_0+z)<0\quad\forall z\in\mathbb{R}^2~~\text{with }|z|\ge R\,,~~\forall \lambda\in[\alpha,\beta]\,.
\end{equation}
If $K_0<0$ the same conclusion holds with $u_0(t)=r_0e^{-2\pi it}$. 
\end{Lemma}

\begin{proof} Let us consider the case $K_0>0$. One has that $L(u_0+z)=L(u_0)=2\pi r_0$ for every $r_0>0$ and $z\in\mathbb{R}^2$. Moreover, by \eqref{KirschLaurain}
\[
G_0(u_0+z)=-\pi r_0^2\quad\text{and}\quad G_1(u_0+z)=-\int_{D_{r_0}(z)}K_1(x,y)\,dx\,dy\quad\forall r_0>0\,,~\forall z\in\mathbb{R}^2
\]
where $D_{r_0}(z)$ is the disc centered at $z$ and with radius $r_0$. Hence, if $\lambda\in[\alpha,\beta]$, 
\[
E_\lambda(u_0+z)\le 2\pi r_0-\alpha \pi r_0^2K_0+\beta\int_{D_{r_0}(z)}K_1(x,y)\,dx\,dy\,.
\]
Since $K_0>0$ we can find $r_0>0$ such that 
\begin{equation}\label{eps0}
2\pi r_0-\alpha \pi r_0^2K_0<0\quad\text{and}\quad 2\pi r_0>\frac{2\pi}{\beta\left\|K\right\|_\infty}\,.
\end{equation}
Moreover, by (K4), there exists $R>0$ such that 
\[
\beta\int_{D_{r_0}(z)}|K_1(x,y)|\,dx\,dy<-\pi r_0+\frac{\alpha \pi r_0^2K_0}{2}\quad\forall z\in\mathbb{R}^2\quad\text{with}\quad|z|\ge R\,.
\]
Then the conclusion follows. Similarly when $K_0<0$.
\end{proof}

Now we introduce the mountain pass level and we prove some some bounds. 

\begin{Proposition}\label{mp2}
Assume that $K\colon\mathbb{R}^{2}\to\mathbb{R}$ is a continuous function satisfying $(K4)$. Let $\alpha,\beta>0$ with $\alpha<\beta$ and let $r_0,R,u_0$ as in Lemma \ref{L:negative}. Set $\overline{u}:=u_0+z_0$ for some fixed $z_0\in\mathbb{R}^2$ with $|z_0|\ge R$, $\Gamma:=\{\gamma\in C([0,1],H^{1})\colon\gamma(0)=0\,,~\gamma(1)=\overline{u}\}$ and 
\begin{equation}\label{mpEl}
c_{\lambda}:=\inf_{\gamma\in\Gamma}\max_{s\in[0,1]}E_{\lambda}(\gamma(s))\,.
\end{equation}
Then 
\[
\frac{\pi}{\lambda\left\|K\right\|_{\infty}}\le c_{\lambda}\le\frac{\pi}{\lambda\left|K_0\right|}~~\forall\lambda\in[\alpha,\beta]\,.
\]
\end{Proposition}

\begin{proof} We assume $K_0>0$. The lower bound for $c_\lambda$ can proved as in Proposition \ref{P:mp}, using Theorem \ref{ii} and Lemma \ref{L:negative}. Let us check the upper bound. Let
\[
\varepsilon_{0}:=-2\pi r_{0}+\alpha\pi r_{0}^{2}K_{0}
\]
and observe that $\varepsilon_{0}>0$, by \eqref{eps0}. 
For every $\varepsilon\in(0,\varepsilon_{0})$ there exists $R_\varepsilon\ge R$ such that 
\[
\beta\int_{D_{r_0}(z)}|K_{1}(x,y)|\,dx\,dy<\varepsilon\quad\forall z\in\mathbb{R}^2\quad\text{with}\quad |z|\ge R_\varepsilon\,.
\]
Fix $z_\varepsilon\in \mathbb{R}^2$ with $|z_\varepsilon|\ge R_\varepsilon$. 
Let $\gamma=\gamma_1\cup\gamma_2\cup\gamma_3$ with $\gamma_1=[0,z_\varepsilon]$, $\gamma_2(s)=z_\varepsilon+su_0$ ($s\in[0,1]$) and $\gamma_3(s)=g(s)+u_0$ with $g\in C([0,1],\mathbb{R}^2)$ such that $g(0)=z_\varepsilon$, $g(1)=z_{0}$ and $|g(s)|\ge R$ for every $s\in[0,1]$. We compute 
\begin{gather}\label{Eg1}
E_{\lambda}(\gamma_{1}(s))=0\\\label{Eg2}
E_{\lambda}(\gamma_{2}(s))=2\pi r_{0}s-\lambda K_{0}s\pi r_{0}^{2}-\lambda\int_{D_{sr_{0}}(z_{\varepsilon})}K_{1}(x,y)\,dx\,dy\\\label{Eg3}
E_{\lambda}(\gamma_{3}(s))=E_{0,\lambda}(u_{0})-\lambda\int_{D_{sr_{0}}(g(s))}K_{1}(x,y)\,dx\,dy
\end{gather}
for every $s\in[0,1]$. In particular
\[
E_{\lambda}(\gamma_{3}(s))\le 2\pi r_{0}-\lambda\pi r_{0}^{2}K_{0}+\lambda\int_{D_{r_{0}}(g(s))}|K_{1}(x,y)|\,dx\,dy\le 2\pi r_{0}-\alpha\pi r_{0}^{2}K_{0}+\varepsilon< 0
\]
because $\varepsilon<\varepsilon_{0}$. Hence $\gamma\in\Gamma$ and 
\[\begin{split}
\max_{s\in[0,1]}E_{\lambda}(\gamma(s))&=\max_{s\in[0,1]}E_{\lambda}(\gamma_{2}(s))\\
&\le \max_{s\in[0,1]}\left[2\pi r_{0}s-\lambda K_{0}s^{2}\pi r_{0}^{2}+\beta\int_{D_{r_{0}}(z_{\varepsilon})}|K_{1}(x,y)|\,dx\,dy\right]\\
&\le \frac{\pi}{\lambda K_{0}}+\varepsilon\,.
\end{split}\]
Therefore, by definition of $c_{\lambda}$ and by the arbitrariness of $\varepsilon>0$,  the upper bound for $c_{\lambda}$ is proved. The case $K_{0}<0$ is similar. 
\end{proof}

Let $\Lambda_{\alpha,\beta}$ be defined as in \eqref{LambdaD}. Arguing as for Proposition \ref{P:PS-bound}, we have
\begin{Lemma}\label{L:PSseq}
If $\lambda\in\Lambda_{\alpha,\beta}$ then there exists a PS sequence $(u_{n})\subset H^{1}\setminus\mathbb{R}^{2}$ for $E_{\lambda}$ at level $c_{\lambda}$, with $\sup_{n}L(u_{n})<\infty$.
\end{Lemma}

As a next step, we analize the behaviour of the PS sequences given by Lemma \ref{L:PSseq}.

\begin{Lemma}
Fix $\lambda\in\Lambda_{\alpha,\beta}$ and let $(u_n)$ be the PS sequence for $E_{\lambda}$ given by Lemma \ref{L:PSseq}.  For every $n\in\mathbb{N}$, set
\[
\overline{u}_n:=\int_0^1u_n(t)\,dt\quad\text{and}\quad v_n:=u_n-\overline{u}_n\,.
\]
\begin{itemize}
\item[$(i)$]
If $\liminf|\overline{u}_{n}|<\infty$, then there exists $u\in H^1\setminus\mathbb{R}^2$ such that, up to a subsequence, $u_n\to u$ strongly in $H^1$ and $u$ is a $\lambda K$-loop.
\item[$(ii)$]
If $|\overline{u}_{n}|\to\infty$, then there exists $v\in H^1\setminus\mathbb{R}^2$ such that, up to a subsequence, $v_n\to v$ strongly in $H^1$ and $v$ is a $\lambda K_0$-loop. 
\end{itemize}
Moreover, if $c_{\lambda}<\frac\pi{\left|\lambda K_{0}\right|}$, then only the case $(i)$ can occur. 
\end{Lemma} 
\begin{proof} (i) If $\liminf|\overline{u}_{n}|<\infty$, then $(u_{n})$ admits a bounded subsequence in $H^{1}$ and then a subsequence weakly converging to some $u$. In this case, the conclusion follows from Lemma \ref{L:compactness}. 

(ii) Observe that $(v_{n})$ is bounded in $H^{1}$ and
\[
E'_{\lambda}(u_{n})[h]=E'_{0,\lambda}(v_{n})[h]+\lambda\int_{0}^{1}K_{1}(v_{n}+\overline{u}_{n})h\cdot i\dot{u}_{n}\,dt
\]
and 
\begin{equation}\label{restoE'}
\int_{0}^{1}|K_{1}(v_{n}+\overline{u}_{n})\dot{u}_{n}|\,dt\to 0
\end{equation}
Indeed, observe that $(\dot{u}_{n})$ is bounded in $L^{1}$, $(v_{n})$ is bounded in $C^{0}$, and $|\overline{u}_{n}|\to\infty$. Hence \eqref{restoE'} follows from $(K4)$.
Therefore, by the embedding $H^{1}\hookrightarrow C^{0}$,
\[
\|E'_{0,\lambda}(v_{n})\|\le\|E'_{\lambda}(u_{n})\|+C\int_{0}^{1}|K_{1}(v_{n}+\overline{u}_{n})\dot{u}_{n}|\,dt
\]
which yields $\|E'_{0,\lambda}(v_{n})\|\to 0$. Moreover, for a subsequence, $v_{n}$ converges weakly in $H^{1}$ to some $v\in H^{1}$. By Lemma \ref{L:compactness} (applied to $E_{0,\lambda}$), $v_{n}\to v$ strongly in $H^{1}$. In particular, $v$ is a 1-periodic solution of
\begin{equation}
\label{v-eq}
\ddot v=L(v)\lambda K_{0}i\dot v
\end{equation}
with 
\begin{equation}
\label{null-ave}
\int_{0}^{1}v\,dt=0
\end{equation}
because $\int_{0}^{1}v_{n}\,dt=0$ for every $n$. The case $v\in\mathbb{R}^2$ cannot occur, otherwise $L(u_{n})=L(v_{n})\to L(v)=0$. By \eqref{ii} also $G(u_{n})\to 0$ and thus $E_{\lambda}(u_{n})\to 0$, which is impossible, since $c_{\lambda}\ne 0$. Hence, $v$ is a $\lambda K_{0}$-loop. More precisely, by an elementary argument, one can easily obtain that the nonconstant 1-periodic solutions of \eqref{v-eq}--\eqref{null-ave} are 
\[
v(t)=\frac{e^{2\pi i jt+it_0}}{|\lambda K_0|}
\]
with $j\in\mathbb{Z}\setminus\{0\}$ and $t_0\in\mathbb{R}$. In particular the sign of $j$ must be equal to the sign of $\lambda K_0$ and 
\begin{equation}\label{en0}
E_{0,\lambda}(v)=\frac{j\pi}{\lambda K_0}\,.
\end{equation}
Using \eqref{KirschLaurain}, noticing that $\mathrm{range}(u_n)\subset D_\rho(\overline{u}_n)$ for some $\rho>0$ independent of $n$ and recalling that the index of a curve $\mathcal{C}_u=\mathrm{range}(u)$ vanishes on the unbounded component of $\mathbb{R}^2\setminus\mathcal{C}_u$, we can estimate
\[
G_1(u_n)=-\int_{\mathbb{R}^2}\mathrm{Ind}_{u_n}(z)K_1(z)\,dz=-\int_{D_\rho(\overline{u}_n)}\mathrm{Ind}_{u_n}(z)K_1(z)\,dz\,.
\]
Then, by \eqref{Rado},
\[
\left|G_1(u_n)\right|\le\sup_{z\in D_\rho(\overline{u}_n)}|K_1(z)|\int_{\mathbb{R}^2}|\mathrm{Ind}_{v_n}(z)|\,dz\le\frac{1}{4\pi}\sup_{z\in D_\rho(\overline{u}_n)}|K_1(z)|\left(\int_0^1|\dot{v}_n|\,dt\right)^2=o(1)
\]
thanks to (K4). Thence, since $v_n\to v$ in $H^1$ and by \eqref{en0},
\[
c_\lambda=E_\lambda(u_n)+o(1)=E_{0,\lambda}(v_n)+\lambda G_1(u_n)+o(1)=E_{0,\lambda}(v)=\frac{j\pi}{\lambda K_0}\,.
\]
This proves the last statement. 
\end{proof}

Finally, we show:

\begin{Lemma}\label{LK5}
If (K5) holds, then $c_{\lambda}<\frac\pi{\left|\lambda K_{0}\right|}$.
\end{Lemma}

\begin{proof} Suppose $K_{0}>0$. Fix an interval $[\alpha,\beta]\subset(0,\infty)$ containing $\lambda$. Let $r_{0},R>0$ be given by Lemma \ref{L:negative}. 
Let $\Omega:=\{z\in\mathbb{R}^{2}\colon |z|>r\,,~\mathrm{arg}(z)\in I\}$, where $r>0$ and $I\subset\mathbb{S}^{1}$ are given according to the assumption (K5). Let $z_{1}\in\mathbb{R}^{2}$ be such that $D_{r_{0}}(z_{1})\subset\Omega\cap\{z\in\mathbb{R}^{2}\colon|z|\ge R\}$. Finally, let $\gamma=\gamma_1\cup\gamma_2\cup\gamma_3$ with $\gamma_1=[0,z_1]$, $\gamma_2(s)=z_1+su_0$ ($s\in[0,1]$) and $\gamma_3(s)=g(s)+u_0$ with $g\in C([0,1],\mathbb{R}^2)$ such that $g(0)=z_1$, $g(1)=z_{0}$ and $|g(s)|\ge R$ for every $s\in[0,1]$. By computations like \eqref{Eg1}--\eqref{Eg3}, one obtains that $\gamma\in\Gamma$ and 
\[
c_{\lambda}\le\max_{s\in[0,1]}E_{\lambda}(\gamma(s))=
\max_{s\in[0,1]}E_{\lambda}(\gamma_{2}(s))<\max_{s\in[0,1]}E_{0,\lambda}(\gamma_{2}(s))=\frac\pi{\lambda K_{0}}
\]
where the strict inequality is due to the fact that, by $(K5)$, $K_{1}>0$ in $D_{sr_{1}}(z_{0})$ for all $s\in[0,1]$. 
\end{proof}

\noindent
\emph{Proof of Theorem \ref{T:2}.} Fix an interval $[\alpha,\beta]\subset(0,\infty)$. By the Denjoy-Young-Saks Theorem \cite{Sak64}, the set $\Lambda_{\alpha,\beta}$ (defined in \eqref{LambdaD})) has full measure in $[\alpha,\beta]$. Then, by Lemmata \ref{L:PSseq}--\ref{LK5}, for a.e. $\lambda\in[\alpha,\beta]$ there exists a $\lambda K$-loop. Similarly, when $[\alpha,\beta]\subset(-\infty,0)$. The conclusion follows from the arbitrariness of $[\alpha,\beta]$ in $\mathbb{R}\setminus\{0\}$.\hfill$\square$
\medskip

One can recognize that in fact we proved:

\begin{Theorem}\label{T3}
Let $K\colon\mathbb{R}^{2}\to\mathbb{R}$ be a continuous function satisfying $(K4)$. For every $\lambda\in\mathbb{R}\setminus\{0\}$ let $E_{\lambda}$ and $c_{\lambda}$ be defined as in \eqref{E0l} and \eqref{mpEl}, respectively. Then 
\[
c_{\lambda}\le\frac{\pi}{\lambda\left|K_0\right|}~~\forall\lambda\in\mathbb{R}\setminus\{0\}
\]
and $c_{\lambda}$ is a critical value of $E_{\lambda}$ (hence a $\lambda K$-loop exists) for a.e. $\lambda$ such that $c_{\lambda}<\frac{\pi}{\lambda\left|K_0\right|}$. \end{Theorem}

\end{document}